\documentclass[10pt,a4paper]{article}
\usepackage[utf8]{inputenc}
\usepackage{amsmath}
\usepackage{pifont}
\usepackage{bm}
\usepackage[american]{babel}
\usepackage[all]{xy}
\usepackage{enumitem}
\usepackage{anysize}
\usepackage[dvipsnames]{xcolor}
\usepackage{graphicx}
\usepackage{stmaryrd}
\usepackage{amsfonts,amstext,amssymb,amsmath,amsthm,pspicture,bigints}

\usepackage{fullpage}
\usepackage{moreverb}
\usepackage{pifont}
\usepackage{pst-all}
\usepackage[left=2cm,right=2cm,top=2cm,bottom=2cm]{geometry}
\usepackage{esint}
\usepackage{fourier-orns}
\usepackage{mathrsfs}
\usepackage{array}

\usepackage{hyperref}

\usepackage{authblk}
\newcolumntype{C}[1]{>{\centering\arraybackslash}b{#1}}
\newcolumntype{R}[1]{>{\raggedleft\arraybackslash}b{#1}}
\newcolumntype{L}[1]{>{\raggedright\arraybackslash}b{#1}}
\newcolumntype{M}[1]{>{\centering}m{#1}}
\newtheorem{theo}{Theorem}[section]
\newtheorem{defin}{Definition}[section]
\newtheorem{lem}{Lemma}[section]
\newtheorem{prop}{Proposition}[section]

\newtheorem{remark}{Remark}[section]

\usepackage{bbm}
\numberwithin{equation}{section}
\usepackage{tikz}

\usepackage{subfigure}

\usepackage{pgfplots}
\usepackage{amsfonts,amsmath}
\usepackage{tikz-3dplot}
\pgfplotsset{compat=newest}
\usetikzlibrary{calc}
\usetikzlibrary{quotes,angles}
\usepackage{pgfplots}
\pgfplotsset{compat=1.12}
\usetikzlibrary{math}

\date{}
\title{Filamentation near monotone zonal vortex caps}
\author{Gian Marco Marin\qquad Emeric Roulley}
\begin{document}
	\maketitle
	\begin{abstract}
		We study the Euler equations on a rotating unit sphere, focusing on the dynamics of vortex caps. Leveraging the $L^1$-stability of monotone, longitude-independent profiles, we demonstrate that certain ill-prepared initial data within the vortex cap class exhibit an instability characterized by the growth of the interface perimeter. These configurations are nearly equivalent in area to a zonal vortex cap but are perturbed by a localized latitudinal bump. By comparing the longitudinal flows at points along the zonal interface and within the bump region, we track the induced stretching and capture the underlying instability mechanism.
	\end{abstract}
\tableofcontents
\section{Introduction}
We begin this document by presenting the barotropic model, namely the Euler equations on the rotating unit sphere, and discuss some relevant related literature. Then, we present the notion of vortex cap solutions. Finally, we state our main instability result of vortex caps "near" the zonal stationary ones.
\subsection{The model and associated literature}
The unit sphere is denoted
	$$\mathbb{S}^2\triangleq\left\{\mathbf{x}\triangleq(x_1,x_2,x_3)\in\mathbb{R}^3\quad\textnormal{s.t.}\quad |\mathbf{x}|_3^2\triangleq x_1^2+x_2^2+x_3^2=1\right\}.$$
    The sphere $\mathbb{S}^2$ is assumed to rotate uniformly around the north pole at constant speed $\gamma\in\mathbb{R}.$ We consider a fluid on this sphere and denote $u,\omega$ its velocity field and vorticity, respectively. We also define the absolute vorticity as follows
\begin{equation}\label{def abso-vort}
    \zeta(t,\mathbf{x})\triangleq\omega(t,\mathbf{x})-2\gamma x_3.
\end{equation}
    The fluid is supposed to follow the Eulerian evolution law and therefore solves the following set of equations  called \textit{barotropic model}, see \cite[Sec. 13.4.1]{HH04}. This model plays a central role in geophysical fluid dynamics as it serves as a fundamental tool for understanding large-scale atmospheric flows, aiding in the simulation of atmospheric behavior for both Earth and other planets, with applications ranging from weather prediction and hurricane dynamics to planetary climate analysis.
	\begin{equation}\label{barotropic model}
		\begin{cases}
			\partial_t\zeta+u\cdot\nabla\zeta=0,\\
			u=\nabla^{\perp}G[\omega],\\
			\Delta G[\omega]=\omega,
		\end{cases}
	\end{equation}
    where $\Delta,$ $\nabla^{\perp}$ are the Laplacian and orthogonal gradient on the sphere, while $G[\omega]$ is the stream function defined through the integral relation, see \cite{BD15},
\begin{equation}\label{Green kernel}
    G[\omega](t,\mathbf{x})=\int_{\mathbb{S}^2}G(\mathbf{x},\mathbf{x}')\omega(t,\mathbf{y})d\sigma(\mathbf{x}'),\qquad G(\mathbf{x},\mathbf{y})\triangleq\frac{1}{2\pi}\log\left(\frac{|\mathbf{x}-\mathbf{y}|_{3}}{2}\right)-\frac{\log(2)}{4\pi}\cdot
\end{equation}
In the above expression, $\sigma$ denotes the classical surface measure on the unit $2$-sphere. Since $\mathbb{S}^2$ is a compact manifold, the vorticity and absolute vorticity are subject to the Gauss constraint
\begin{equation}\label{Gauss constraint}
    \int_{\mathbb{S}^2}\zeta(t,\mathbf{x})d\sigma(\mathbf{x})=\int_{\mathbb{S}^2}\omega(t,\mathbf{x})d\sigma(\mathbf{x})=0.
\end{equation}
Let us mention that
$$0=\partial_{t}\zeta+u\cdot\nabla\zeta=\partial_{t}\omega+u\cdot\nabla(\omega-2\gamma x_3)$$
and, compared to the classical 2D Euler equations, the additional term $-2\gamma u\cdot\nabla x_3$ is the Coriolis force due to the rotation of the sphere. Let us now discuss a bit more about the manifold structure of $\mathbb{S}^2.$ It is seen as a smooth manifold with atlas given by the following two local charts $\psi_1,\psi_2:(0,\pi)\times(0,2\pi)\to\mathbb{R}^3$
	\begin{align*}
		\psi_1(\theta,\varphi)&\triangleq\big(\sin(\theta)\cos(\varphi),\sin(\theta)\sin(\varphi),\cos(\theta)\big),\\
		\psi_2(\vartheta,\phi)&\triangleq\big(-\sin(\vartheta)\cos(\phi),-\cos(\vartheta),-\sin(\vartheta)\sin(\phi)\big).
	\end{align*}
    Given a function $f:\mathbb{S}^2\to\mathbb{R}$, we denote
    $$\mathtt{f}(\theta,\varphi)\triangleq f\big(\psi_1(\theta,\varphi)\big).$$
    In the sequel, we shall identify both functions passing from Cartesian to spherical coordinates keeping the same notation $f=\mathtt{f}$. In particular, in the local chart $\psi_1,$ one has
    \begin{equation}\label{x3cost}
        x_3=\cos(\theta).
    \end{equation}
	In what follows, we may restrict our discussion to the local chart $\psi_1$ where the variables are the co-latitude $\theta$ and the longitude $\varphi$, respectively. Nevertheless, one can follow the argument while working in the local chart $\psi_2$ and then cover the whole sphere. The manifold $\mathbb{S}^2$ is endowed with a Riemannian structure where the metric is given (in the chart $\psi_1$) by
	\begin{equation}\label{def metric}
	    \mathtt{g}_{\mathbb{S}^2}(\theta,\varphi)\triangleq d\theta\otimes d\theta+\sin^2(\theta)d\varphi\otimes d\varphi
	\end{equation}
	and the associated Riemannian volume writes (still in the chart $\psi_1$)
	$$\sigma=\sin(\theta) d\theta\wedge d\varphi.$$
    Therefore, the integration on the sphere is
    $$\int_{\mathbb{S}^2}f(\mathbf{x})d\sigma(\mathbf{x})=\int_{0}^{2\pi}\int_{0}^{\pi}f(\theta,\varphi)\sin(\theta)d\theta d\varphi.$$
	The north pole is $N\triangleq\{\theta=0\}$ while the south one is $S\triangleq\{\theta=\pi\}.$ Throughout the document, we shall use the notation: given $\theta_{*}\in(0,\pi)$, the parallel at the co-latitude $\theta_*$ is the set
    $$\{\theta=\theta_{*}\}\triangleq\big\{\psi_1(\theta_*,\varphi),\quad\varphi\in[0,2\pi]\big\}.$$
    For any $\mathbf{x}=\psi_1(\theta,\varphi)\in\mathbb{S}^2$ the tangent space $T_{\mathbf{x}}\mathbb{S}^2$ admits the orthonormal basis $(\mathbf{e}_{\theta}(\mathbf{x}),\mathbf{e}_{\varphi}(\mathbf{x}))$ given by
	$$\mathbf{e}_{\theta}(\mathbf{x})\triangleq\partial_{\theta}\psi_1(\theta,\varphi),\qquad\mathbf{e}_{\varphi}(\mathbf{x})\triangleq\frac{\partial_{\varphi}\psi_1(\theta,\varphi)}{\sin(\theta)}\cdot$$
	This allows to give an expression of the gradient and orthogonal gradient in this basis
	$$\nabla f\triangleq\left(\partial_{\theta}f,\frac{\partial_{\varphi}f}{\sin(\theta)}\right)_{(\mathbf{e}_{\theta},\mathbf{e}_{\varphi})},\qquad\nabla^{\perp}f\triangleq\left(-\frac{\partial_{\varphi}f}{\sin(\theta)},\partial_{\theta}f\right)_{(\mathbf{e}_{\theta},\mathbf{e}_{\varphi})}.$$
    The Laplacian expresses as
    \begin{equation}\label{Laplacian sphere}
        \Delta f(\theta,\varphi)=\frac{1}{\sin(\theta)}\partial_{\theta}\big[\sin(\theta)\partial_{\theta}f(\theta,\varphi)\big]+\frac{1}{\sin^2(\theta)}\partial_{\varphi}^2f(\theta,\varphi).
    \end{equation}
    In order to be well-defined, the set of equations \eqref{barotropic model} is supplemented by the following impermeability conditions at the poles: denoting $u=(u_{\theta},u_{\varphi})_{(\mathbf{e}_{\theta},\mathbf{e}_{\varphi})}$, then
	\begin{equation}\label{impermea}
	    \forall t\geqslant0,\quad\forall\varphi\in\mathbb{T},\quad u_{\theta}(t,0,\varphi)=0=u_{\theta}(t,\pi,\varphi).
	\end{equation}

    The system \eqref{barotropic model} is invariant under rotation around the vertical axis (passing through the poles). As a consequence, any longitude-independant profile $\zeta(\theta,\varphi)=\zeta(\theta)$ is a trivial stationary solution called \textit{zonal solution}. Any function $G$ solving
    \begin{equation}\label{stateq}
        \Delta G(\theta,\varphi)-2\gamma\cos(\theta)=F\big(G(\theta,\varphi)\big)
    \end{equation}
    provides the stream function of a stationary solution of \eqref{barotropic model}. Notice that the reciprocal is not true in general. In their work \cite{CG22}, Constantin and Germain established that solutions to equation \eqref{stateq} where $F'>-6$ are necessarily zonal (up to a rotation). In addition they are stable in $H^2(\mathbb{S}^2),$ provided that $F'<0.$ The threshold value $-6$ is significant, as it corresponds to the second eigenvalue of the Laplace–Beltrami operator. Some special zonal solutions called Rossby-Haurwitz are given by a stream function in the form
    $$G_n(\theta)=\beta Y_n^0(\theta)+\frac{2\gamma}{n(n+1)-2}\cos(\theta),\qquad\beta\in\mathbb{R}^*$$
    where $Y_n^0$ is the spherical harmonic. The integer $n\in\mathbb{N}$ is called the degree of the Rossby-Haurwitz solution. Reference \cite{CG22} also explored both local and global bifurcations of non-zonal solutions to equation \eqref{stateq}, emerging from Rossby–Haurwitz waves. The authors show that zonal Rossby–Haurwitz solutions of degree 2 are stable in the space $H^2(\mathbb{S}^2),$ while more general non-zonal solutions of the same type are unstable in this setting. More recently, Cao, Wang, and Zuo \cite{CWZ23} extended the stability analysis of degree 2 Rossby–Haurwitz waves to the $L^p(\mathbb{S}^2)$ spaces for $p\in(1,\infty)$. Furthermore, and of interest in our analysis, Caprino and Marchioro \cite{CM88} addressed $L^1$-Lyapunov stability for monotonic zonal vorticities within $L^p(\mathbb{S}^2)$, for $p\in(2,\infty)$. A precise statement is given later in Theorem \ref{thm stab CM}. This work deals with vortex cap solutions, which are special weak solutions where the absolute vorticity is uniform on domains forming a partition of the sphere. These solutions together with their linear and nonlinear stability has been intensively studied in the physics literature, see for instance \cite{CDG24,DP92,DP93,KSS18,KS21}, while their rigorous mathematical description (briefly recalled here in Section \ref{sec VC}) has been presented in \cite{GHR23}. In this latter, the third author proved the emergence of small amplitude uniformly rotating vortex cap solutions bifurcating from the zonal caps. Let us also mention that recently, some global in time solutions were obtained by desingularizing point vortex configurations like the symmetric pairs \cite{CLW24} or the  Von K\'arm\'an streets \cite{SZ24,SZ24-1}. At last, we highlight that in \cite{CDG24}, Dritschel, Constantin and Germain numerically studied the onset of the filamentation on both plane (near circular patches) and sphere (near zonal caps). This work serves as a one motivation for the present paper. 
    \subsection{Vortex cap solutions}\label{sec VC}
 The mathematical notion of vortex caps has been introduced in \cite{GHR23}. These are weak solutions to \eqref{barotropic model} that are piecewise constant absolute vorticities. They constitute the equivalent to the classical planar vortex patches and one of the main differences with the latter is the Gauss constraint \eqref{Gauss constraint}, which brings more rigidity and therefore complexifies the analysis with respect to the planar case.
 
	\begin{defin}[Vortex Cap]\label{vcap}
		Let $N\in\mathbb{N}\setminus\{0,1\}$ and consider the following partition of the sphere 
		$$\mathbb{S}^2=\bigsqcup_{k=1}^{N}A_k,\qquad\sigma(A_k)>0$$
		such that each boundary intersection (called interface) is diffeomorphic to a circle, i.e.
		$$\forall k\in\llbracket1,N-1\rrbracket,\quad\Gamma_k\triangleq A_{k}\cap A_{k+1}\cong \mathbb{S}^1.$$ Let $\omega_1,...,\omega_N\in\mathbb{R}$ such that 
        \begin{equation}\label{Gauss initial}
            \sum_{k=1}^{N}\omega_k\sigma(A_k)=0\qquad\textnormal{and}\qquad\forall k\in\llbracket1,N-1\rrbracket,\quad\omega_k\neq\omega_{k+1}.
        \end{equation}
        Let us consider an initial condition in the form $$\zeta_0\triangleq\sum_{k=1}^{N}\omega_k\mathbf{1}_{A_k}.$$
        Due to the transport structure \eqref{barotropic model} and the logarithmic singularity of the Green kernel \eqref{Green kernel}, the Yudovich theory applies and provides the existence and uniqueness of a Lagragian weak solution to \eqref{barotropic model} called \textit{vortex cap solution}, namely
        $$\zeta(t,\cdot)=\sum_{k=1}^{N}\omega_k\mathbf{1}_{A_k(t)}, \qquad A_k(t)\triangleq\phi(t,A_k),$$
        where $(t,\mathbf{x})\mapsto\phi(t,\mathbf{x})$ is the flow map associated with the vector field $u,$ that is
        $$\forall\mathbf{x}\in\mathbb{S}^2,\quad\partial_{t}\phi(t,\mathbf{x})=u\big(t,\phi(t,\mathbf{x})\big),\qquad\phi(0,\mathbf{x})=\mathbf{x}.$$
	\end{defin}
	\begin{remark}  The first condition in \eqref{Gauss initial} corresponds to the Gauss constraint \eqref{Gauss constraint} of the initial datum. Since the velocity field $u$ is divergence-free, then, for any $t\geqslant0,$ the flow map $\phi(t,\cdot)$ is measure preserving. Consequently, one has
        $$\forall k\in\llbracket 1,N\rrbracket,\quad\sigma\big(A_k(t)\big)=\sigma\big(\phi(t,A_k)\big)=\sigma(A_k)$$
        and therefore $\zeta(t,\cdot)$ also satisfies the Gauss constraint.
	\end{remark}
	This work makes a particular focus on trivial vortex cap solutions provided by the zonal caps, namely 
    $$\zeta_{\star}(\theta)=\omega_1\mathbf{1}_{0\leqslant\theta<\theta_1}+\omega_2\mathbf{1}_{\theta_1\leqslant \theta<\theta_2}+...+\omega_N\mathbf{1}_{\theta_{N-1}\leqslant \theta<\pi},$$
    with
    \begin{equation}\label{Gausszonal}
		\theta_0\triangleq0<\theta_1<\theta_2<...<\theta_N-1<\theta_N\triangleq\pi\qquad\textnormal{and}\qquad\sum_{k=1}^{N}\omega_k\big(\cos(\theta_{k})-\cos(\theta_{k-1})\big)=0.
	\end{equation}
    The second condition in \eqref{Gausszonal} is nothing but the Gauss constraint.
\subsection{Main result and strategy of proof}
    The study of filamentation/growth of perimeter is a very important topic in fluid mechanics in presence of free interface. Several results in this direction were obtained near steady solutions. We refer the reader for instance to \cite{CJ21,CJ22} for the planar vortex patch case and to \cite{CJ23} near the Hill vortex. We shall now present our main new result concerning the spherical geometry. The purpose of this study is to prove the following instability result in the vortex cap class "close to" the monotone zonal caps.
\begin{theo}\label{thm filamentation}
\textbf{$($Filamentation near monotone zonal vortex caps$)$}\\
    Let $N\in\mathbb{N}\setminus\{0,1\}$, $\mathtt{M}\geqslant1$ and 
    \begin{equation}\label{order thetas}
        0\triangleq\theta_0<\theta_1<\theta_2<...<\theta_{N-1}<\theta_N\triangleq\pi.
    \end{equation}
    Consider the monotone zonal cap 
    $$\zeta_{\star}(\theta)=\omega_1\mathbf{1}_{0\leqslant\theta<\theta_1}+\omega_2\mathbf{1}_{\theta_1\leqslant\theta<\theta_2}+...+\omega_N\mathbf{1}_{\theta_{N-1}\leqslant\theta\leqslant\pi},$$
    with
    \begin{equation}\label{monotone cond}
        \begin{cases}
        \omega_1<\omega_2<...<\omega_N, & \textnormal{if }\gamma\leqslant0,\\
        \omega_1>\omega_2>...>\omega_N, & \textnormal{if }\gamma\geqslant0
    \end{cases}
    \end{equation}
    and
    $$\sum_{k=1}^{N}\omega_k\big(\cos(\theta_k)-\cos(\theta_{k-1})\big)=0.$$
    There exists $\mu_0>0$, such that for all $\mu\in(-\mu_0,\mu_0),$ there exist $\kappa\triangleq\kappa(\mu)>0$ and $T_0\triangleq T_0(\mu)>0$ such that for all $T>T_0,$ there exists $\overline{\delta}\triangleq\overline{\delta}(\mu,\mathtt{M},T)>0$ such that the following holds:\\
    Given a vortex cap solution $t\mapsto\zeta(t,\cdot)$ of \eqref{barotropic model} with initial condition $\zeta_0$ satisfying the bounds
    $$\|\zeta_0-\zeta_{\star}\|_{L^1(\mathbb{S}^2)}<\overline{\delta}\qquad\textnormal{and}\qquad\|\zeta_0\|_{L^{\infty}(\mathbb{S}^2)}\leqslant\mathtt{M},$$
    and admitting an initial interface $\Gamma(0)$ such that
    $$\exists k_0\in\llbracket1,N-1\rrbracket,\quad\Gamma(0)\cap\{\theta=\theta_{k_0}\}\neq\varnothing\qquad\textnormal{and}\qquad\Gamma(0)\cap\{\theta=\theta_{k_0}+\mu\}\neq\varnothing,$$
    then the corresponding interface evolution $t\mapsto\Gamma(t)\triangleq\phi\big(t,\Gamma(0)\big)$ satisfies
    $$\sup_{0\leqslant t\leqslant T}\textnormal{Length}\big(\Gamma(t)\big)\geqslant\kappa(T-T_0).$$
\end{theo}
\begin{remark}
    Let us make the following remarks concerning the Theorem \ref{thm filamentation}.
    \begin{enumerate}
        \item The condition \eqref{monotone cond} ensures that
        $$\omega_{\star}(\theta)=\zeta_{\star}(\theta)+2\gamma\cos(\theta)$$
        is monotone. This is fundamental in our analysis based on the $L^1$-stability result Theorem \ref{thm stab CM} below.
        \item The value $|\mu|$ is a priori small (less than $\mu_0$), which corresponds to a small latitudinal bump. But if there is a large thin bump in particular, there is a point $\mathbf{x}_1$ as in the statement.
        \item The boundedness hypothesis $\|\zeta_0\|_{L^{\infty}(\mathbb{S}^2)}\leqslant\mathtt{M}$ is required along the proof but is not so much restrictive since $\mathtt{M}$ can be taken large and cover a huge set of initial data. The small parameter $\overline{\delta}(\mu,\mathtt{M},T)$ shrinks to zero as $\mathtt{M}\to\infty$ and $T\to\infty.$
        \item The picture of the theorem is that we take an initial cap $\zeta_0$ which is $L^1$-close enough to the zonal cap $\zeta_{\star}$ but with a little bump. Then, this bump will create the filamentation, see Figure \ref{figure}.
    \end{enumerate}
\end{remark}

    \begin{figure}
        \begin{minipage}{.45\linewidth}
        \includegraphics[width=9cm]{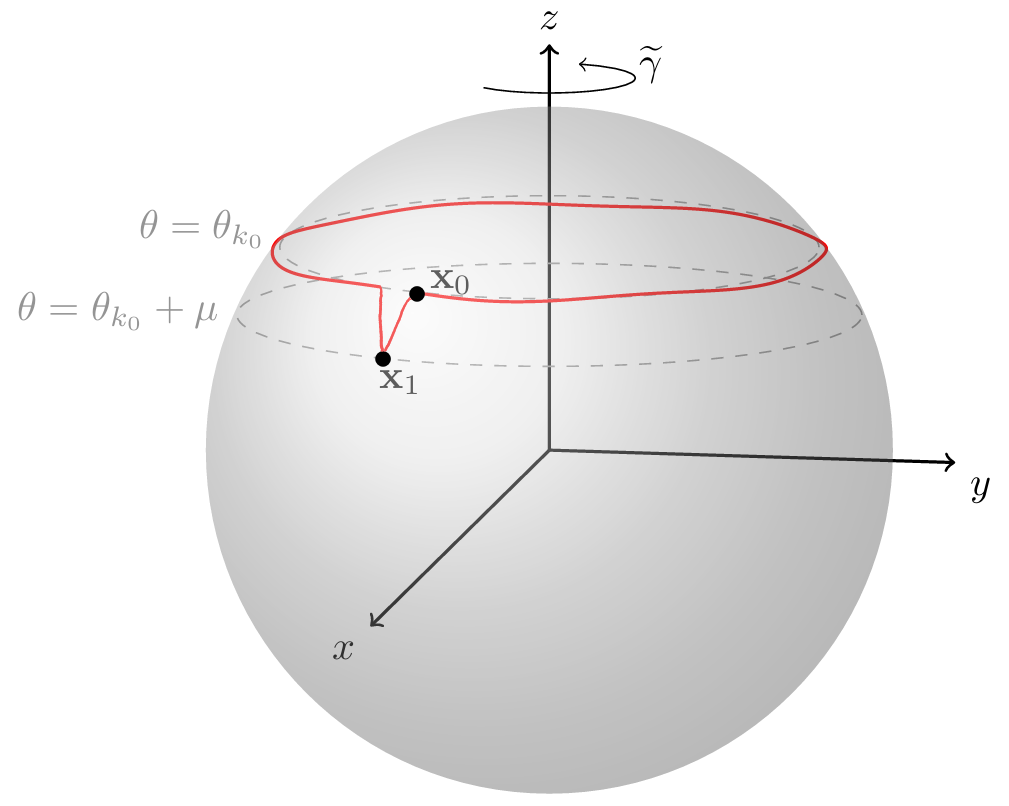}
        \end{minipage}
        \qquad\quad
        \begin{minipage}{.45\linewidth}
            \centering\includegraphics[width=8.5cm]{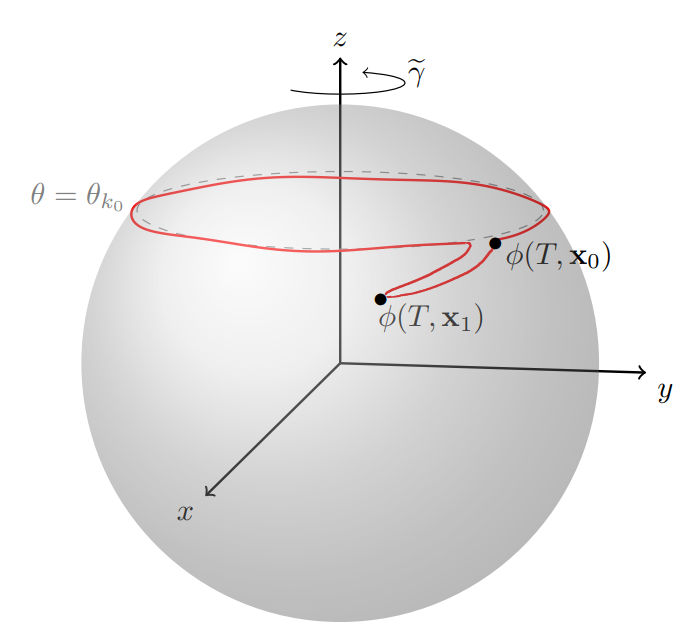}
        \end{minipage}
        \caption{Illustration of the filamentation Theorem \ref{thm filamentation}.}
        \label{figure}
    \end{figure}

Let us now give some key steps of the proof of Theorem \ref{thm filamentation}. We strongly make use of the following area stability result of monotone zonal profiles due to Caprino and Marchioro \cite[Thm. 1.1]{CM88}. The monotonicity is required since they use variational arguments with rearrangement functions.
\begin{theo}[$L^1$-stability of zonal profiles on the rotating sphere]\label{thm stab CM}
    Let $p>2$ and $\omega_{\star}\in L^p(\mathbb{S}^2)$ be a monotone zonal solution of \eqref{barotropic model}. Then, for any $\varepsilon>0,$ there exists $\delta>0$ such that for any $\omega_0\in L^p(\mathbb{S}^2)$ with
    $$\|\omega_0-\omega_{\star}\|_{L^1(\mathbb{S}^2)}<\delta,$$
    then the solution $t\mapsto\omega(t)$ of \eqref{barotropic model} with initial datum $\omega_0$ satisfies
    $$\sup_{t\geqslant 0}\|\omega(t,\cdot)-\omega_{\star}\|_{L^1(\mathbb{S}^2)}<\varepsilon.$$
\end{theo}
\begin{remark}
    According to \eqref{def abso-vort}, one has
    $$\forall t\geqslant0,\quad\omega(t,\cdot)-\omega_{\star}=\zeta(t,\cdot)-\zeta_{\star}.$$
    Hence, the theorem holds replacing $\|\omega_0-\omega_{\star}\|_{L^1(\mathbb{S}^2)}$ by $\|\zeta_0-\zeta_{\star}\|_{L^1(\mathbb{S}^2)}$ and $\|\omega(t,\cdot)-\omega_{\star}\|_{L^1(\mathbb{S}^2)}$ by $\|\zeta(t,\cdot)-\zeta_{\star}\|_{L^1(\mathbb{S}^2)}$. Nevertheless, the monotone condition in the Theorem \ref{thm stab CM} hits the actual vorticity $\omega$ and that is the reason why we imposed the condition \eqref{monotone cond} in the Theorem \ref{thm filamentation}.
\end{remark}
In Lemma \ref{lem around barotropic model}, we prove that for $u=\nabla^{\perp}G[\omega],$ we have
$$\|u\|_{L^{\infty}}\lesssim\sqrt{\|\omega\|_{L^{\infty}(\mathbb{S}^2)}\|\omega\|_{L^1(\mathbb{S}^2)}}.$$
Applying this estimate to the difference $u-u_{\star}$ and exploiting the Theorem \ref{thm stab CM}, we are able to prove that
\begin{equation}\label{stab intro}
    \|\zeta_0-\zeta_{\star}\|_{L^1(\mathbb{S}^2)}\ll1\qquad\Rightarrow\qquad\sup_{0\leqslant t\leqslant T}\|u(t,\cdot)-u_{\star}\|_{L^{\infty}(\mathbb{S}^2)}\ll1.
\end{equation}
We consider $\Theta$ and $\Phi$ the co-latitude and longitude flows so that
$$\phi(t,\mathbf{x})=\psi_1\big(\Theta(t,\mathbf{x}),\Phi(t,\mathbf{x})\big).$$
The longitude is also lifted to $\mathbb{R}$ in order to follow the perimeter growth. For a zonal flow $\zeta_{\star}$, we prove that for any $\mathbf{x}=\psi_1(\theta_{\mathbf{x}},\varphi_{\mathbf{x}}),$ we have
$$\Theta_{\star}(t,\mathbf{x})=\Theta(0,\mathbf{x})\triangleq\theta_{\mathbf{x}},\qquad\Phi_{\star}(t,\mathbf{x})=\dot{\Phi}_{\star}(\theta_{\mathbf{x}})t+\varphi_{\mathbf{x}},\qquad\dot{\Phi}_{\star}(\theta)\triangleq\frac{\partial_{\theta}G[\zeta_{\star}](\theta)}{\sin(\theta)}+\gamma.$$
Using in particular \eqref{stab intro}, we show in Proposition \ref{prop confinement} that far from the poles, the approximate flow follows the zonal linear dynamics on the interval $[0,T]$
$$\Theta(t,\mathbf{x})\approx\theta_{\mathbf{x}},\qquad\textnormal{and}\qquad\Phi(t,\mathbf{x})\approx\dot{\Phi}_{\star}(\theta_{\mathbf{x}})t+\varphi_{\mathbf{x}}.$$
We take an initial condition $\zeta_0$ with interface $\Gamma(0)$ such that
$$\mathbf{x}_0,\mathbf{x}_1\in\Gamma(0),\qquad\theta_{\mathbf{x}_0}=\theta_{k_0},\qquad\theta_{\mathbf{x}_1}=\theta_{\mathbf{x}_1}(\mu)=\theta_{k_0}+\mu.$$
Using the expression of the length of a curve on the sphere and the approximate flow dynamics, we get the bound
$$\textnormal{Length}\big(\gamma_T([0,1])\big)\gtrsim\left|\Phi(T,\mathbf{x}_1)-\Phi(T,\mathbf{x}_0)\right|\approx\left|\dot{\Phi}_{\star}\big(\theta_{\mathbf{x}_1}(\mu)\big)-\dot{\Phi}_{\star}(\theta_{\mathbf{x}_0})\right|T.$$
Exploiting the explicit formulation for $\partial_{\theta}G[\zeta_{\star}]$ when $\zeta_{\star}$ is a monotone zonal vortex cap (see Lemma \ref{lem zonal}), we are able to prove in Lemma \ref{lem alpha non zero} that for $|\mu|\ll1,$ we get
$$\left|\dot{\Phi}_{\star}\big(\theta_{\mathbf{x}_1}(\mu)\big)-\dot{\Phi}_{\star}(\theta_{\mathbf{x}_0})\right|\neq0.$$
This concludes the desired result.\\

\noindent\textbf{Plan of the paper :} In Section \ref{sec properties}, we state some general results that concern our equation and that can be used in other contexts. The Section \ref{sec proof} is devoted to the proof of our main result.\\

\noindent\textbf{Acknowledgments :} The work of Emeric Roulley is supported by PRIN 2020 "Hamiltonian and Dispersive PDEs", project number: 2020XB3EFL.  The authors would like to thank Alberto Maspero for useful discussions and for making this collaboration possible.

\section{Preliminary properties}\label{sec properties}
This section gathers some general technical results for our model. These are quite general and might be useful in other contexts regarding the barotropic model.
\subsection{Generic properties of the barotropic model and analysis of zonal caps}
Here we give some basic properties of the stream function. We also discuss the particular zonal case of interest in our study.
\begin{lem}\label{lem around barotropic model}
The following properties hold true.
\begin{enumerate}[label=(\roman*)]
    \item For any $\mathbf{x}=(x_1,x_2,x_3)\in\mathbb{S}^2$ and $t\geqslant0,$ we have
    \begin{equation}\label{Gomg-Gzeta}
        G[\omega](t,\mathbf{x})=G[\zeta](t,\mathbf{x})-\gamma x_3.
    \end{equation}
    \item One has
    \begin{equation}\label{delta Gz}
        \Delta G[\zeta]=\zeta.
    \end{equation} 
\end{enumerate}
    
\end{lem}
\begin{proof}
    $(i)$ Observe from \eqref{def abso-vort} that 
    \begin{equation}\label{interm1}
        G[\omega](t,\mathbf{x})=G[\zeta+2\gamma x_3](t,\mathbf{x})=G[\zeta](t,\mathbf{x})+2\gamma G[x_3].
    \end{equation}
    Now, since $x_3=\cos(\theta)$ is zonal, then according to \cite[Lem. 1.2]{GHR23}, $G[x_3]$ is also zonal and solves the equation $\Delta G[x_3]=x_3$ which becomes via \eqref{Laplacian sphere} and \eqref{x3cost},
    $$\frac{1}{\sin(\theta)}\partial_{\theta}\big(\sin(\theta)\partial_{\theta}G[x_3](\theta)\big)=\cos(\theta).$$
    Hence,
    $$\partial_{\theta}\big(\sin(\theta)\partial_{\theta}G[x_3](\theta)\big)=\sin(\theta)\cos(\theta)=\tfrac{1}{2}\sin(2\theta).$$
    Integrating this relation implies the existence of a constant $C\in\mathbb{R}$ such that
    $$\partial_{\theta}G[x_3](\theta)=\frac{C-\cos(2\theta)}{4\sin(\theta)}\cdot$$
    Since the flow is zonal, the impermeability condition \eqref{impermea} implies that there is no velocity at the pole, which gives
    $$\lim_{\theta\to0^+}\partial_{\theta}G[x_3](\theta)=0.$$
    As a consequence, $C=1$ and 
    $$\partial_{\theta}G[x_3](\theta)=\frac{1-\cos(2\theta)}{4\sin(\theta)}=\tfrac{1}{2}\sin(\theta).$$
    Integrating this relation gives the existence of $C'\in\mathbb{R}$ such that
    $$G[x_3](\theta)=C'-\tfrac{1}{2}\cos(\theta).$$
    We find the value of $C'$ thanks to the zero mean condition for the stream function (which is a consequence of the Gauss constraint, see \cite{CLW24}),
    $$0=\int_{\mathbb{S}^2}G[x_3]d\sigma(\mathbf{x})=4\pi C'-\pi\int_{0}^{\pi}\cos(\theta)\sin(\theta)d\theta=4\pi C'.$$
    Therefore, $C'=0$ and 
    \begin{equation}\label{Gx3}
        G[x_3](\theta)=-\tfrac{1}{2}\cos(\theta),\qquad\textnormal{i.e.}\qquad G[x_3]=-\frac{x_3}{2}\cdot
    \end{equation}
    Plugging \eqref{Gx3} into \eqref{interm1} gives the desired result.\\
    $(ii)$ The relation \eqref{Gx3} implies
    $$\Delta x_3=-2x_3.$$
    Consequently, using the point $(i)$, the third equation in \eqref{barotropic model} and \eqref{def abso-vort}, we infer that for any $\mathbf{x}=(x_1,x_2,x_3)\in\mathbb{S}^2$ and $t\geqslant0,$
    \begin{align*}
        \Delta G[\zeta](t,\mathbf{x})&=\Delta\big(G[\omega](t,\mathbf{x})+\gamma x_3\big)\\
        &=\Delta G[\omega](t,\mathbf{x})+\gamma\Delta x_3\\
        &=\omega(t,\mathbf{x})-2\gamma x_3\\
        &=\zeta(t,\mathbf{x}).
    \end{align*}
    This concludes the proof of Lemma \ref{lem around barotropic model}.
\end{proof}
Now we turn to the analysis of the stream function associated with a zonal vortex cap. The result is the following.
\begin{lem}\label{lem zonal}
        Let $N\in\mathbb{N}\setminus\{0,1\}$ and $0\triangleq\theta_0<\theta_1<\theta_2<...<\theta_{N-1}<\theta_N\triangleq\pi.$ Consider the zonal cap 
    $$\zeta_{\star}(\theta)=\omega_1\mathbf{1}_{0\leqslant\theta<\theta_1}+\omega_2\mathbf{1}_{\theta_1\leqslant\theta<\theta_2}+...+\omega_N\mathbf{1}_{\theta_{N-1}\leqslant\theta\leqslant\pi},$$
    with
    $$\sum_{k=1}^{N}\omega_k\big(\cos(\theta_k)-\cos(\theta_{k-1})\big)=0.$$
    Then, the stream function $G[\zeta_{\star}]$ is of class $C^1$ on $(0,\pi)$ and satisfies for any $k_{0}\in\llbracket 1,N\rrbracket$ and $\theta\in[\theta_{k_0-1},\theta_{k_0})\setminus\{0\},$ 
    \begin{equation}\label{deriv Gzetastar}
        \partial_{\theta}G[\zeta_{\star}](\theta)=\frac{1}{\sin(\theta)}\left(\sum_{k=1}^{k_0-1}\omega_{k}\big(\cos(\theta_{k-1})-\cos(\theta_k)\big)+\omega_{k_0}\big(\cos(\theta_{k_0-1})-\cos(\theta)\big)\right).
    \end{equation} 
\end{lem}
\begin{remark}\label{remark differentiability}
    With the expression \eqref{deriv Gzetastar}, we see that the function $\partial_{\theta}G[\zeta_{\star}]$ is continuous on $(0,\pi)$ and differentiable on $(0,\pi)\setminus\{\theta_1,\theta_2,...,\theta_{N-1}\}.$
\end{remark}
\begin{proof}
    According to \eqref{delta Gz} and \eqref{Laplacian sphere}, we have
    $$\partial_{\theta}\big(\sin(\theta)\partial_{\theta}G[\zeta_{\star}](\theta)\big)=\sin(\theta)\left(\omega_1\mathbf{1}_{0\leqslant\theta<\theta_1}+\omega_2\mathbf{1}_{\theta_1\leqslant\theta<\theta_2}+...+\omega_N\mathbf{1}_{\theta_{N-1}\leqslant\theta\leqslant\pi}\right).$$
    Assume that $\theta\in[\theta_{k_0-1},\theta_{k_0})\setminus\{0\}$ for some $k_0\in\llbracket 1, N\rrbracket.$ Then, integrating the previous relation leads to the existence of a constant $C\in\mathbb{R}$ such that
    $$\partial_{\theta}G[\zeta_{\star}](\theta)=\frac{1}{\sin(\theta)}\left(\sum_{k=1}^{k_0-1}\omega_{k}\big(\cos(\theta_{k-1})-\cos(\theta_k)\big)+\omega_{k_0}\big(\cos(\theta_{k_0-1})-\cos(\theta)\big)+C\right).$$
    The constant $C$ is independant of $k_0$ and since the flow is zonal, there is no velocity at the pole. This implies that $C=0,$ which gives the desired result.
\end{proof}
\subsection{Velocity estimates}
Now we prove a technical lemma used along the paper. In particular, we prove that if the vorticity in bounded, then the velocity field is also bounded.
\begin{lem}\label{lem convol}
The following properties hold true.
\begin{enumerate}[label=(\roman*)]
    \item Given $f\in L^{\infty}(\mathbb{S}^2)\subset L^{1}(\mathbb{S}^2),$ we have
    $$\left\|\frac{1}{|\cdot|_{\mathbb{R}^3}}\ast f\right\|_{L^{\infty}(\mathbb{S}^2)}\lesssim\sqrt{\|f\|_{L^{\infty}(\mathbb{S}^2)}\|f\|_{L^1(\mathbb{S}^2)}}.$$
    \item Given $\omega\in L^{\infty}(\mathbb{S}^2)$ and $u=\nabla^{\perp}G[\omega]$, we have
    $$\|u\|_{L^{\infty}(\mathbb{S}^2)}\lesssim\sqrt{\|\omega\|_{L^{\infty}(\mathbb{S}^2)}\|\omega\|_{L^1(\mathbb{S}^2)}}.$$
    In particular, if $\omega$ is the vorticity of a vortex cap, then
    $$\sup_{t\geqslant0}\|u(t,\cdot)\|_{L^{\infty}(\mathbb{S}^2)}\lesssim\sup_{t\geqslant0}\|\omega(t,\cdot)\|_{L^{\infty}(\mathbb{S}^2)}<\infty.$$
    \item Let $\omega_{\star}\in L^{\infty}(\mathbb{S}^2)$ be a monotone zonal solution of \eqref{barotropic model}. Then for any $\varepsilon>0,$ there exists $\delta_1\triangleq\delta_1(\varepsilon)>0$ such that for any $\omega_0\in L^{\infty}(\mathbb{S}^2)\setminus\{0\},$ with
    $$\|\omega_0-\omega_{\star}\|_{L^1(\mathbb{S}^2)}<\delta_1,$$
    then, denoting $u=\nabla^{\perp}G[\omega]$ (resp. $u_{\star}=\nabla^{\perp}G[\omega_{\star}]$) the velocity field associated with the solution $t\mapsto\omega(t,\cdot)$ (resp. $\omega_{\star}$),
    we have
    $$\sup_{t\geqslant0}\|u(t,\cdot)-u_{\star}\|_{L^{\infty}(\mathbb{S}^2)}<\varepsilon\sup_{t\geqslant0}\sqrt{\|\omega(t,\cdot)-\omega_{\star}\|_{L^{\infty}(\mathbb{S}^2)}}.$$
    In particular, this holds for $\omega_0,\omega_{\star}$ vortex caps.
\end{enumerate}
\end{lem}
\begin{proof} $(i)$ We assume $f\neq0$ otherwise the result is trivial. Recall that, denoting $\mathtt{d}_{\mathbb{S}^2}(\mathbf{x},\mathbf{y})$ the geodesic distance between $\mathbf{x}\in\mathbb{S}^2$ and $\mathbf{y}\in\mathbb{S}^2,$ we have
    $$|\mathbf{x}-\mathbf{y}|_{\mathbb{R}^3}=2\sin\left(\frac{\mathtt{d}_{\mathbb{S}^2}(\mathbf{x},\mathbf{y})}{2}\right).$$
    Given $\mathbf{x}\in\mathbb{S}^2$ and $r>0$, we denote
    $$B_{\mathbb{S}^2}(\mathbf{x},r)\triangleq\left\lbrace\mathbf{y}\in\mathbb{S}^2\quad\textnormal{s.t.}\quad\mathtt{d}_{\mathbb{S}^2}(\mathbf{x},\mathbf{y})<r\right\rbrace.$$
    Then, for any $\mathbf{x}\in\mathbb{S}^2$ and any arbitrary $r\in(0,\pi]$, we can bound
    \begin{align*}
        \left|\left(\frac{1}{|\cdot|_{\mathbb{R}^3}}\ast f\right)(\mathbf{x})\right|&=\int_{\mathbb{S}^2}\frac{f(\mathbf{y})}{|\mathbf{x}-\mathbf{y}|_{\mathbb{R}^3}}d\sigma(\mathbf{y})\\
        &\leqslant\int_{B_{\mathbb{S}^2}(\mathbf{x},r)}\frac{|f(\mathbf{y})|}{|\mathbf{x}-\mathbf{y}|_{\mathbb{R}^3}}d\sigma(\mathbf{y})+\int_{\mathbb{S}^2\setminus B_{\mathbb{S}^2}(\mathbf{x},r)}\frac{|f(\mathbf{y})|}{|\mathbf{x}-\mathbf{y}|_{\mathbb{R}^3}}d\sigma(\mathbf{y})\\
        &\leqslant\int_{B_{\mathbb{S}^2}(N,r)}\frac{|f\circ\mathcal{R}_{\mathbf{x}}(\mathbf{y})|}{|N-\mathbf{y}|_{\mathbb{R}^3}}d\sigma(\mathbf{y})+\int_{\mathbb{S}^2\setminus B_{\mathbb{S}^2}(N,r)}\frac{|f\circ\mathcal{R}_{\mathbf{x}}(\mathbf{y})|}{|N-\mathbf{y}|_{\mathbb{R}^3}}d\sigma(\mathbf{y}),
    \end{align*}
    where $\mathcal{R}_{\mathbf{x}}\in SO_3(\mathbb{R})$ is a rotation sending the north pole $N$ onto the point $\mathbf{x}$. A direct bound gives
    $$\int_{\mathbb{S}^2\setminus B_{\mathbb{S}^2}(N,r)}\frac{|f\circ\mathcal{R}_{\mathbf{x}}(\mathbf{y})|}{|N-\mathbf{y}|_{\mathbb{R}^3}}d\sigma(\mathbf{y})\leqslant\frac{1}{2\sin\left(\frac{r}{2}\right)}\|f\circ\mathcal{R}_{\mathbf{x}}\|_{L^1(\mathbb{S}^2)}=\frac{1}{2\sin\left(\frac{r}{2}\right)}\|f\|_{L^1(\mathbb{S}^2)}.$$
    Then, passing to spherical coordinates, we get
    \begin{align*}
        \int_{B_{\mathbb{S}^2}(N,r)}\frac{|f\circ\mathcal{R}_{\mathbf{x}}(\mathbf{y})|}{|N-\mathbf{y}|_{\mathbb{R}^3}}d\sigma(\mathbf{y})&\leqslant\|f\circ\mathcal{R}_{\mathbf{x}}\|_{L^{\infty}(\mathbb{S}^2)}\int_{0}^{2\pi}\int_{0}^{r}\frac{\sin(\theta)}{2\sin\left(\frac{\theta}{2}\right)}d\theta d\varphi\\
        &=2\pi\|f\|_{L^{\infty}(\mathbb{S}^2)}\int_{0}^{r}\cos\left(\tfrac{\theta}{2}\right)d\theta\\
        &=4\pi\sin\left(\tfrac{r}{2}\right)\|f\|_{L^{\infty}(\mathbb{S}^2)}.
    \end{align*}
    We deduce that
    $$\left|\left(\frac{1}{|\cdot|_{\mathbb{R}^3}}\ast f\right)(\mathbf{x})\right|\leqslant4\pi\sin\left(\tfrac{r}{2}\right)\|f\|_{L^{\infty}(\mathbb{S}^2)}+\frac{1}{2\sin\left(\frac{r}{2}\right)}\|f\|_{L^1(\mathbb{S}^2)}.$$
    Choosing 
    $$r\triangleq2\arcsin\left(\sqrt{\frac{\|f\|_{L^1(\mathbb{S}^2)}}{4\pi\|f\|_{L^{\infty}(\mathbb{S}^2)}}}\right),$$
    we obtain the desired bound.\\
    $(ii)$ $\blacktriangleright$ Fix $\mathbf{x}\in\mathbb{S}^2.$ According to \eqref{Green kernel} and \eqref{Gauss constraint}, we can write
    \begin{align*}
        u(\mathbf{x})&=\frac{1}{4\pi}\nabla^{\perp}\left(\int_{\mathbb{S}^2}\log\left(|\mathbf{x}-\mathbf{y}|_{\mathbb{R}^3}^2\right)\omega(\mathbf{y})d\sigma(\mathbf{y})\right)\\
        &=\frac{1}{2\pi}\int_{\mathbb{S}^2}\frac{\langle\mathbf{x}-\mathbf{y},\nabla^{\perp}\mathbf{x}\rangle_{\mathbb{R}^3}}{|\mathbf{x}-\mathbf{y}|_{\mathbb{R}^3}^2}\omega(\mathbf{y})d\sigma(\mathbf{y}).
    \end{align*}
    By Cauchy-Schwarz inequality, we obtain
    $$|u(\mathbf{x})|\lesssim\left(\frac{1}{|\cdot|_{\mathbb{R}^3}}\ast|\omega|\right)(\mathbf{x}).$$
    Applying the point $(i)$ allows to get he desired result.\\
    $\blacktriangleright$ Now we assume that $\omega$ is associated to a vortex cap in the form
    $$\omega(t,\mathbf{x})=\sum_{k=1}^{N}\omega_k\mathbf{1}_{A_k(t)}(\mathbf{x})+2\gamma x_3.$$
    In particular,
    $$\sup_{t\geqslant0}\|\omega(t,\cdot)\|_{L^{\infty}(\mathbb{S}^2)}\leqslant2|\gamma|+\sum_{k=1}^{N}|\omega_k|.$$
    Recalling that $\|\cdot\|_{L^1(\mathbb{S}^2)}\leqslant4\pi\|\cdot\|_{L^{\infty}(\mathbb{S}^2)},$ we get
    $$\sup_{t\geqslant0}\|u(t,\cdot)\|_{L^{\infty}(\mathbb{S}^2)}\lesssim\sup_{t\geqslant0}\|\omega(t,\cdot)\|_{L^{\infty}(\mathbb{S}^2)}<\infty.$$
    $(iii)$ Observe that
    $$u-u_{\star}=\nabla^{\perp}G[\omega]-\nabla^{\perp}G[\omega_{\star}]=\nabla^{\perp}G[\omega-\omega_{\star}].$$
    Fix $\varepsilon>0.$ In view of the first point, there exists a universal constant $C>0$ such that for any $t\geqslant0,$
    \begin{align*}
        \|u(t,\cdot)-u_{\star}\|_{L^{\infty}(\mathbb{S}^2)}&\leqslant C\sqrt{\|\omega(t,\cdot)-\omega_{\star}\|_{L^{\infty}(\mathbb{S}^2)}\|\omega(t,\cdot)-\omega_{\star}\|_{L^{1}(\mathbb{S}^2)}}.
    \end{align*}
    We apply the stability Theorem \ref{thm stab CM} with
    $$\delta_1(\varepsilon)=\delta\left(\frac{\varepsilon^2}{C^2}\right).$$
    This concludes the proof of Lemma \ref{lem convol}.
\end{proof}

\section{Proof of the filamentation}\label{sec proof}
This section is devoted to the proof of our main Theorem \ref{thm filamentation}. We first describe the general time evolution of the co-latitude and longitude components of the flow map. As an application, we study the case of zonal flows where the evolution is linear in time in the longitude variable while remaining at a fixed co-latitude. Then, we provide the approximate dynamics of an approximate zonal vortex cap flow. Finally, we apply this later fact to prove our main result.
\subsection{Co-latitudinal and longitudinal evolutions}
Recall that $\phi$ denotes the flow map associated with the velocity field $u$, namely
$$\forall t\geqslant0,\quad\forall\mathbf{x}\in\mathbb{S}^2,\quad\partial_t\phi(t,\mathbf{x})=u\big(t,\phi(t,\mathbf{x})\big)\qquad\textnormal{and}\qquad\phi(0,\mathbf{x})=\mathbf{x}.$$
We denote $\pi_{\theta}$ and $\widetilde{\pi}_{\varphi}$ the co-latitudinal and lifted longitudinal projections defined by
$$\pi_{\theta}(\theta,\varphi)=\theta\in(0,\pi)\qquad\textnormal{and}\qquad\widetilde{\pi}_{\varphi}(\theta,\varphi)=\varphi\in\mathbb{R}.$$
Then, we defined the co-latitude and longitude flows by
$$\Theta(t,\mathbf{x})\triangleq\pi_{\theta}\circ\psi_1^{-1}\big(\phi(t,\mathbf{x})\big)\qquad\textnormal{and}\qquad\Phi(t,\mathbf{x})\triangleq\widetilde{\pi}_{\varphi}\circ\psi_1^{-1}\big(\phi(t,\mathbf{x})\big),$$
so that
\begin{equation}\label{flow in local chart}
    \phi(t,\mathbf{x})=\psi_1\big(\Theta(t,\mathbf{x}),\Phi(t,\mathbf{x})\big)=\begin{pmatrix}
    \sin\big(\Theta(t,\mathbf{x})\big)\cos\big(\Phi(t,\mathbf{x})\big)\\
    \sin\big(\Theta(t,\mathbf{x})\big)\sin\big(\Phi(t,\mathbf{x})\big)\\
    \cos\big(\Theta(t,\mathbf{x})\big)
\end{pmatrix}.
\end{equation}
Let us mention that such definitions extend globally with the local chart $\psi_2$.
\subsubsection{General discussion}
In the next proposition, we give the generic time evolution of the co-latitude and longitude flows.
\begin{prop}\label{prop dtTP}
For any $\mathbf{x}\in\mathbb{S}^2$ and $t\geqslant0$ such that $\sin\big(\Theta(t,\mathbf{x})\big)\neq0,$ we have
    $$\partial_{t}\Theta(t,\mathbf{x})=-\frac{\partial_{\varphi}G[\omega]\big(t,\phi(t,\mathbf{x})\big)}{\sin\big(\Theta(t,\mathbf{x})\big)}=u_{\theta}\big(t,\phi(t,\mathbf{x})\big)\qquad\textnormal{and}\qquad\partial_{t}\Phi(t,\mathbf{x})=\frac{\partial_{\theta}G[\omega]\big(\phi(t,\mathbf{x})\big)}{\sin\big(\Theta(t,\mathbf{x})\big)}=\frac{u_{\varphi}\big(t,\phi(t,\mathbf{x})\big)}{\sin\big(\Theta(t,\mathbf{x})\big)}\cdot$$
\end{prop}
\begin{proof}
    The elements $\mathbf{e}_{\theta},\mathbf{e}_{\varphi}$ can be written as vector in $\mathbb{R}^3$ as follows: if $\mathbf{x}=\psi_1(\theta,\varphi),$ then
    $$\mathbf{e}_{\theta}(\mathbf{x})=\partial_{\theta}\psi_1(\theta,\varphi)=\begin{pmatrix}
        \cos(\theta)\cos(\varphi)\\
        \cos(\theta)\sin(\varphi)\\
        -\sin(\theta)
    \end{pmatrix}\qquad\textnormal{and}\qquad\mathbf{e}_{\varphi}(\mathbf{x})=\frac{1}{\sin(\theta)}\partial_{\varphi}\psi_1(\theta,\varphi)=\begin{pmatrix}
        -\sin(\varphi)\\
        \cos(\varphi)\\
        0
    \end{pmatrix}.$$
    Differentiating in time \eqref{flow in local chart}, we get
    \begin{equation}\label{dtphi in basis}
        \begin{aligned}
        u\big(t,\phi(t,\mathbf{x})\big)&=\partial_t\phi(t,\mathbf{x})\\
        &=\partial_{t}\Theta(t,\mathbf{x})\partial_{\theta}\psi_1\big(\Theta(t,\mathbf{x}),\Phi(t,\mathbf{x})\big)+\partial_{t}\Phi(t,\mathbf{x})\partial_{\varphi}\psi_1\big(\Theta(t,\mathbf{x}),\Phi(t,\mathbf{x})\big)\\
    &=\partial_{t}\Theta(t,\mathbf{x})\mathbf{e}_{\theta}\big(\phi(t,\mathbf{x})\big)+\sin\big(\Theta(t,\mathbf{x})\big)\partial_{t}\Phi(t,\mathbf{x})\mathbf{e}_{\varphi}\big(\phi(t,\mathbf{x})\big).
    \end{aligned}
    \end{equation}
    Besides,
    \begin{equation}\label{u in basis}
        \begin{aligned}
        u\big(t,\phi(t,\mathbf{x})\big)&=\nabla^{\perp}G[\omega]\big(t,\phi(t,\mathbf{x})\big)\\
        &=-\frac{\partial_{\varphi}G[\omega]\big(t,\phi(t,\mathbf{x})\big)}{\sin\big(\Theta(t,\mathbf{x})\big)}\mathbf{e}_{\theta}\big(\phi(t,\mathbf{x})\big)+\partial_{\theta}G[\omega]\big(t,\phi(t,\mathbf{x})\big)\mathbf{e}_{\varphi}\big(\phi(t,\mathbf{x})\big).
    \end{aligned}
    \end{equation}
    Comparing \eqref{dtphi in basis} and \eqref{u in basis} gives the desired result.
\end{proof}
In the particular case of a zonal flow $\omega_{\star}=\omega_{\star}(\theta)$, one has according to \cite[Lem. 1.2]{GHR23} that the stream function is longitude-independant $G[\omega_{\star}]=G[\omega_{\star}](\theta).$ Consequently, Proposition \ref{prop dtTP} implies that for any $\mathbf{x}\in\mathbb{S}^2$ and any $t\geqslant0,$
$$\partial_{t}\Theta_{\star}(t,\mathbf{x})=0,\qquad\textnormal{i.e.}\qquad\Theta_{\star}(t,\mathbf{x})=\theta_{\mathbf{x}}\triangleq\pi_{\theta}\circ\psi_1^{-1}(\mathbf{x}).$$
The motion is only longitudinal, namely
$$\phi_{\star}(t,\mathbf{x})\in\{\psi_1(\theta_{\mathbf{x}},\varphi),\quad\varphi\in(0,2\pi)\}.$$
Moreover, the previous remark together with Proposition \ref{prop dtTP} give also
$$\partial_{t}\Phi_{\star}(t,\mathbf{x})=\frac{\partial_{\theta}G[\omega_{\star}](\theta_{\mathbf{x}})}{\sin(\theta_{\mathbf{x}})}\triangleq\dot{\Phi}_{\star}(\theta_{\mathbf{x}}),\qquad\textnormal{i.e.}\qquad\Phi_{\star}(t,\mathbf{x})=\dot{\Phi}_{\star}(\theta_{\mathbf{x}})t+\varphi_{\mathbf{x}},\qquad\varphi_{\mathbf{x}}\triangleq\widetilde{\pi}_{\varphi}\circ\psi_{1}^{-1}(\mathbf{x}).$$
So the motion is longitudinal and grows linearly in that direction. Differentiating \eqref{Gomg-Gzeta} and using \eqref{x3cost} leads to
$$\partial_{\theta}G[\omega_{\star}](\theta)=\partial_{\theta}G[\zeta_{\star}](\theta)+\gamma\sin(\theta).$$
Therefore,
\begin{equation}\label{phidot-dG}
    \dot{\Phi}_{\star}(\theta)=\frac{\partial_{\theta}G[\zeta_{\star}](\theta)}{\sin(\theta)}+\gamma.
\end{equation}
\subsubsection{Flow confinement for vortex cap solutions}
Our purpose here is to describe the approximate dynamics of the flow associated to a vortex cap $L^1$-near a monotone zonal cap. The result reads as follows.
\begin{prop}\label{prop confinement}
    Consider $\zeta_{\star}$ a monotone zonal vortex cap solution of \eqref{barotropic model} as stated in Theorem \ref{thm filamentation}. Let $\mathtt{M}\geqslant1,$ $T_0>0$ and $0<\theta_{\min}<\theta_{\max}<\pi.$ There exists $\xi_0\triangleq\xi_0(T_0,\theta_{\min},\theta_{\max})>0$ such that for any $T\geqslant T_0$ and any $0<\xi\leqslant\xi_0,$ there exists $\delta_2\triangleq\delta_2(\xi,\mathtt{M},T)>0$ such that for any vortex cap solution $t\mapsto\zeta(t,\mathbf{x})$ of \eqref{barotropic model} with initial condition $\zeta_0$ satisfying
    \begin{equation}\label{small init}
        \|\zeta_0-\zeta_{\star}\|_{L^1(\mathbb{S}^2)}<\delta_2\qquad\textnormal{and}\qquad\|\zeta_0\|_{L^{\infty}(\mathbb{S}^2)}\leqslant\mathtt{M},
    \end{equation}
    then, for any $\mathbf{x}=\psi_1(\theta_{\mathbf{x}},\varphi_{\mathbf{x}})\in\mathbb{S}^2$ with
    \begin{equation}\label{col restri}
        \theta_{\min}\leqslant\theta_{\mathbf{x}}\leqslant\theta_{\max},
    \end{equation}
    we have
    $$\sup_{0\leqslant t\leqslant T}\left|\Theta(t,\mathbf{x})-\theta_{\mathbf{x}}\right|\leqslant\xi\qquad\textnormal{and}\qquad\sup_{0\leqslant t\leqslant T}\left|\Phi(t,\mathbf{x})-\big(\dot{\Phi}_{\star}(\theta_{\mathbf{x}})t+\varphi_{\mathbf{x}}\big)\right|\leqslant\frac{\xi}{\min^2\big(\sin(\theta_{\min}),\sin(\theta_{\max})\big)}\cdot$$
\end{prop}
\begin{remark}
\begin{enumerate}
    \item The previous proposition states that the time evolution up to $T$ of the point starting at $\mathbf{x}$ is confined in a $\xi$-strip around the parallel $\{\theta=\theta_{\mathbf{x}}\}.$ Moreover, the longitude evolution follows the linear time growth given by the zonal solution.
    \item The co-latitude restriction \eqref{col restri} is useful to stay far from the poles and handle the division by $\sin$ in the longitudinal flow evolution, cf. Proposition \ref{prop dtTP}. Later we will apply this result to points far from the pole so it is not so much restrictive for our analysis. The numbers $\theta_{\min}$ and $\theta_{\max}$ are arbitrary and will be chosen in the next subsection near the poles. As we will see later in the proof (see \eqref{choice xi0}, $\xi_0$ is very small provided that $\theta_{\min}\approx0$ and/or $\theta_{\max}\approx\pi$ for fixed $T_0$.
\end{enumerate}
    
\end{remark}
\begin{proof} 
Fix $T\geqslant T_0.$ We make the choice
\begin{equation}\label{choice delta2}
    \delta_2(\xi,\mathtt{M},T)\triangleq\delta_1\big(f(\xi,\mathtt{M},T)\big),
\end{equation}
where $f(\xi,\mathtt{M},T)$ has to be precised (with upper bounds) along the proof and $\varepsilon\mapsto\delta_1(\varepsilon)$ is defined in Lemma \ref{lem convol}-$(iii).$ 
Since the solution is Lagrangian, then
\begin{equation}\label{bound AV}
    \sup_{0\leqslant t\leqslant T}\|\zeta(t,\cdot)\|_{L^{\infty}(\mathbb{S}^2)}=\sup_{0\leqslant t\leqslant T}\|\zeta_0\big(\phi^{-1}(t,\cdot)\big)\|_{L^{\infty}(\mathbb{S}^2)}=\|\zeta_0\|_{L^{\infty}(\mathbb{S}^2)}\leqslant\mathtt{M}.
\end{equation}
Then, applying Lemma \ref{lem convol}-$(iii)$ with the smallness assumption \eqref{small init} and the choice \eqref{choice delta2} imply
\begin{equation}\label{smallness velocity}
    \begin{aligned}
        \sup_{0\leqslant t\leqslant T}\|u(t,\cdot)-u_{\star}\|_{L^{\infty}(\mathbb{S}^2)}&\leqslant f(\xi,\mathtt{M},T)\sup_{0\leqslant t\leqslant T}\sqrt{\|\zeta(t,\cdot)-\zeta_{\star}\|_{L^{\infty}(\mathbb{S}^2)}}\\
        &\leqslant f(\xi,\mathtt{M},T)\sqrt{\mathtt{M}+\|\zeta_{\star}\|_{L^{\infty}(\mathbb{S}^2)}}.
    \end{aligned}
\end{equation}
Fix $\mathbf{x}=\psi_1(\theta_{\mathbf{x}},\varphi_{\mathbf{x}})\in\mathbb{S}^2$ with $\theta_{\min}\leqslant\theta_{\mathbf{x}}\leqslant\theta_{\max}$ and $t\in[0,T].$ Let us mention for later purposes that
\begin{equation}\label{LB sinus}
    \sin(\theta_{\mathbf{x}})\geqslant\min_{\theta\in[\theta_{\min},\theta_{\max}]}\sin(\theta)=\min\big(\sin(\theta_{\min}),\sin(\theta_{\max})\big)\triangleq\mathtt{m}(\theta_{\min},\theta_{\max})>0.
\end{equation}
$\blacktriangleright$ \textit{Colatitude confinement :} According to Proposition \ref{prop dtTP}, we can write
\begin{align*}
    \Theta(t,\mathbf{x})-\theta_{\mathbf{x}}&=\Theta(t,\mathbf{x})-\Theta(0,\mathbf{x})\\
    &=\int_{0}^{t}\partial_{t}\Theta(s,\mathbf{x})ds\\
    &=\int_{0}^{t}u_{\theta}\big(s,\phi(s,\mathbf{x})\big)ds.
\end{align*}
Since $u_{\star}$ is zonal, then $(u_{\star})_{\theta}\equiv 0.$ Consequently, taking 
$$f(\xi,\mathtt{M},T)\leqslant\tfrac{\xi}{T}\left(\mathtt{M}+\|\zeta_{\star}\|_{L^{\infty}(\mathbb{S}^2)}\right)^{-\frac{1}{2}},$$
we infer from \eqref{smallness velocity} that
\begin{align*}
    \left|\Theta(t,\mathbf{x})-\theta_{\mathbf{x}}\right|&\leqslant\int_{0}^{t}\left|\big(u_{\theta}-(u_{\star})_{\theta}\big)\big(s,\phi(s,\mathbf{x})\big)\right|ds\\
    &\leqslant T\sup_{0\leqslant t\leqslant T}\|u(t,\cdot)-u_{\star}\|_{L^{\infty}(\mathbb{S}^2)}\\
    &\leqslant\xi.
\end{align*}
$\blacktriangleright$ \textit{Longitude confinement :} Notice that
\begin{equation}\label{diff longitude flow}
    \begin{aligned}
    \Phi(t,\mathbf{x})-\big(\dot{\Phi}_{\star}(\theta_{\mathbf{x}})t+\varphi_{\mathbf{x}}\big)&=\Phi(t,\mathbf{x})-\Phi_{\star}(t,\mathbf{x})\\
    &=\int_{0}^{t}\Big[\partial_{t}\Phi\big(s,\phi(s,\mathbf{x})\big)-\partial_{t}\Phi_{\star}\big(s,\phi_{\star}(s,\mathbf{x})\big)\Big]ds.
\end{aligned}
\end{equation}
Take $s\in[0,t].$ According to Proposition \ref{prop dtTP}, we can write
\begin{equation}\label{e D123}
    \begin{aligned}
    \left|\partial_{t}\Phi\big(s,\phi(s,\mathbf{x})\big)-\partial_{t}\Phi_{\star}\big(s,\phi_{\star}(s,\mathbf{x})\big)\right|&=\left|\frac{1}{\sin\big(\Theta(s,\mathbf{x})\big)}u_{\varphi}\big(t,\phi(s,\mathbf{x})\big)-\frac{1}{\sin(\theta_{\mathbf{x}})}(u_{\star})_{\varphi}\big(\phi_{\star}(s,\mathbf{x})\big)\right|\\
    &\leqslant\mathcal{D}_1(s)+\mathcal{D}_2(s)+\mathcal{D}_3(s),
    \end{aligned}
\end{equation}
    where
    \begin{align*}
        \mathcal{D}_1(s)&\triangleq\left|u_{\varphi}\big(s,\phi(s,\mathbf{x})\big)\right|\left|\frac{1}{\sin\big(\Theta(s,\mathbf{x})\big)}-\frac{1}{\sin(\theta_{\mathbf{x}})}\right|,\\
        \mathcal{D}_2(s)&\triangleq\frac{1}{\sin(\theta_{\mathbf{x}})}\left|u_{\varphi}\big(s,\phi(s,\mathbf{x})\big)-(u_{\star})_{\varphi}\big(\phi(s,\mathbf{x})\big)\right|,\\
        \mathcal{D}_3(s)&\triangleq\frac{1}{\sin(\theta_{\mathbf{x}})}\left|(u_{\star})_{\varphi}\big(\phi(s,\mathbf{x})\big)-(u_{\star})_{\varphi}\big(\phi_{\star}(s,\mathbf{x})\big)\right|.
    \end{align*}
    \ding{226} \textit{Estimate of $\mathcal{D}_1(s)$ :} First observe that
    \begin{equation}\label{estim D1-1}
        \mathcal{D}_1(s)\leqslant\frac{\displaystyle\sup_{0\leqslant t\leqslant T}\|u(t,\cdot)\|_{L^{\infty}(\mathbb{S}^2)}}{\sin\big(\Theta(s,\mathbf{x})\big)\sin(\theta_{\mathbf{x}})}\left|\sin\big(\Theta(t,\mathbf{x})\big)-\sin(\theta_{\mathbf{x}})\right|.
    \end{equation}
    Applying Lemma \ref{lem around barotropic model}-$(ii)$ together with \eqref{bound AV}, we have the existence of a universal constant $C>0$ such that
    \begin{equation}\label{bound sup u}
        \displaystyle\sup_{0\leqslant t\leqslant T}\|u(t,\cdot)\|_{L^{\infty}(\mathbb{S}^2)}\leqslant C\left(\sup_{0\leqslant t\leqslant T}\|\zeta_{\star}(t,\cdot)\|_{L^{\infty}(\mathbb{S}^2)}+2|\gamma|\right)\leqslant C\left(\mathtt{M}+2|\gamma|\right).
    \end{equation}
    Using the fact the $\sin$ is $1$-Lipschitz and the proceeding as in the previous point with 
    $$f(\xi,\mathtt{M},T)\leqslant\frac{\xi}{6T^2C(\mathtt{M}+2|\gamma|)}\left(\mathtt{M}+\|\zeta_{\star}\|_{L^{\infty}(\mathbb{S}^2)}\right)^{-\frac{1}{2}},$$ 
    we get
    \begin{equation}\label{Lip sin}
        \left|\sin\big(\Theta(t,\mathbf{x})\big)-\sin(\theta_{\mathbf{x}})\right|\leqslant\left|\Theta(t,\mathbf{x})-\theta_{\mathbf{x}}\right|\leqslant\frac{\xi}{6TC(\mathtt{M}+2|\gamma|)}\cdot
    \end{equation}
    In particular, for 
    \begin{equation}\label{choice xi0}
        \xi_0=\xi_0(T_0,\theta_{\min},\theta_{\max})\triangleq3CT_0\mathtt{m}(\theta_{\min},\theta_{\max})>0,
    \end{equation}
    we get (since $\mathtt{M}\geqslant1$)
    \begin{equation}\label{LB sin}
        \begin{aligned}
            \sin\big(\Theta(t,\mathbf{x})\big)&\geqslant\sin(\theta_{\mathbf{x}})-\frac{\xi_0}{6CT_0(\mathtt{M}+2|\gamma|)}\\
            &\geqslant\mathtt{m}(\theta_{\min},\theta_{\max})-\frac{\xi_0}{6CT_0}\\
            &\geqslant\frac{\mathtt{m}(\theta_{\min},\theta_{\max})}{2}\cdot
        \end{aligned}
    \end{equation}
    Putting together \eqref{estim D1-1}, \eqref{LB sinus}, \eqref{bound sup u}, \eqref{Lip sin} and \eqref{LB sin}, we infer
    \begin{equation}\label{e D1}
        \mathcal{D}_1(t)\leqslant\frac{\xi}{3T\mathtt{m}^2(\theta_{\min},\theta_{\max})}\cdot
    \end{equation}
    \ding{226} \textit{Estimate of $\mathcal{D}_2(s)$ :} Using \eqref{smallness velocity} with 
    $$f(\xi,\mathtt{M},T)\leqslant\frac{\xi}{3T}\left(\mathtt{M}+\|\zeta_{\star}\|_{L^{\infty}(\mathbb{S}^2)}\right)^{-\frac{1}{2}},$$
    one immediately gets
    \begin{equation}\label{e D2}
        \mathcal{D}_2(s)\leqslant\frac{1}{\sin(\theta_{\mathbf{x}})}\sup_{0\leqslant t\leqslant T}\|u(t,\cdot)-u_{\star}\|_{L^{\infty}(\mathbb{S}^2)}\leqslant\frac{\xi}{3T\mathtt{m}^2(\theta_{\min},\theta_{\max})}\cdot
    \end{equation}
    \ding{226} \textit{Estimate of $\mathcal{D}_3(s)$ :} By construction,
    \begin{align*}
        \mathcal{D}_3(s)=\frac{1}{\sin(\theta_{\mathbf{x}})}\left|\partial_{\theta}G[\omega_{\star}]\big(\Theta(s,\mathbf{x})\big)-\partial_{\theta}G[\omega_{\star}]\big(\theta_{\mathbf{x}}\big)\right|.
    \end{align*}
    Without loss of generality, we can assume that $\Theta(s,\mathbf{x})\neq\theta_{\mathbf{x}}$, for the bound being trivial otherwise. Recalling the Remark \ref{remark differentiability}, the function $\partial_{\theta}G[\omega_{\star}]$ is continuous on the segment $[\Theta(s,\mathbf{x}),\theta_{\mathbf{x}}]$ and differentiable on the open interval $(\Theta(s,\mathbf{x}),\theta_{\mathbf{x}})$. By mean value Theorem, we can find a constant $M_{\star}(T)\geqslant0$ (depending on $T$ but uniform in $s$) such that
    $$\left|\partial_{\theta}G[\omega_{\star}]\big(\Theta(s,\mathbf{x})\big)-\partial_{\theta}G[\omega_{\star}]\big(\theta_{\mathbf{x}}\big)\right|\leqslant M_{\star}(T)\left|\Theta(s,\mathbf{x})-\theta_{\mathbf{x}}\right|.$$
    Then, proceeding as in the first point with $f(\xi,T)\leqslant\frac{\xi}{3T^2M_{\star}(T)},$ we obtain
    \begin{equation}\label{e D3}
        \mathcal{D}_3(s)\leqslant\frac{\xi}{3T\mathtt{m}^2(\theta_{\min},\theta_{\max})}\cdot
    \end{equation}
    Putting together \eqref{e D123}, \eqref{e D1}, \eqref{e D2} and \eqref{e D3} yields
    \begin{equation}\label{bound diff dtphi}
        \left|\partial_{t}\Phi\big(s,\phi(s,\mathbf{x})\big)-\partial_{t}\Phi_{\star}\big(s,\phi_{\star}(s,\mathbf{x})\big)\right|\leqslant\frac{\xi}{T\mathtt{m}^2(\theta_{\min},\theta_{\max})}\cdot
    \end{equation}
    Combining \eqref{diff longitude flow} and \eqref{bound diff dtphi} gives the desired result. This achieves the proof of Proposition \ref{prop confinement}.
\end{proof}
\begin{remark}
   In the previous proof, the number $f(\xi,\mathtt{M},T)$ decays like $\mathtt{M}^{-\frac{3}{2}}T^{-2}$ as $\mathtt{M}\to\infty$ and $T\to\infty.$
\end{remark}
\subsection{Stretching argument}
Through out this section, fix $\zeta_{\star}$ a monotone zonal vortex cap solution of \eqref{barotropic model} as in Theorem \ref{thm filamentation}. We consider $\mathtt{M}\geqslant1$ and a vortex cap solution
$t\mapsto\zeta(t,\cdot)$ of \eqref{barotropic model} with initial datum $\zeta_0$ satisfying
$$\|\zeta_0\|_{L^{\infty}(\mathbb{S}^2)}\leqslant\mathtt{M}$$
and admitting an interface $\Gamma(0)$ with, for some $0<|\mu|\ll1$,
$$\Gamma(0)\cap\{\theta=\theta_{k_0}\}\neq\varnothing,\qquad\textnormal{and}\qquad\Gamma(0)\cap\{\theta=\theta_{k_0}+\mu\}\neq\varnothing.$$
This provides the existence of $\mathbf{x}_0,\mathbf{x}_1\in\Gamma(0)$ such that
$$\theta_{\mathbf{x}_0}=\theta_{k_0},\qquad\textnormal{and}\qquad\theta_{\mathbf{x}_1}=\theta_{\mathbf{x}_1}(\mu)\triangleq\theta_{k_0}+\mu.$$
Without loss of generality, we can always assume that
\begin{equation}\label{confined init curve}
    \forall\mathbf{x}\in\Gamma(0)\setminus\{\mathbf{x}_0,\mathbf{x}_1\},\quad|\theta_{k_0}-\theta_{\mathbf{x}}|<|\mu|.
\end{equation}
\begin{lem}\label{lem alpha non zero}
    There exists $\mu_0>0$ such that for any $\mu\in(-\mu_0,\mu_0),$
    $$\dot{\Phi}_{\star}\big(\theta_{\mathbf{x}_1}(\mu)\big)-\dot{\Phi}_{\star}(\theta_{\mathbf{x}_0})\neq0,$$
    where $\dot{\Phi}_{\star}$ is defined in \eqref{phidot-dG}.
\end{lem}
\begin{proof}
    We perform Taylor expansions. For a fixed $\theta\in\mathbb{R},$ we have
    $$\cos(\theta+h)\underset{h\to0}{=}\cos(\theta)-\sin(\theta)h-\tfrac{1}{2}\cos(\theta)h^2+o(h^2)$$
    and
    \begin{align*}
        \frac{1}{\sin^2(\theta+h)}&\underset{h\to0}{=}\frac{1}{\left(\sin(\theta)+\cos(\theta)h-\frac{1}{2}\sin(\theta)h^2+o(h^2)\right)^2}\\
        &\underset{h\to0}{=}\frac{1}{\sin^2(\theta)+2\sin(\theta)\cos(\theta)h+\left(\cos^2(\theta)-\sin^2(\theta)\right)h^2+o(h^2)}\\
        &\underset{h\to0}{=}\frac{1}{\sin^2(\theta)}\cdot\frac{1}{1+2\cot(\theta)h+\left(\cot^2(\theta)-1\right)h^2+o(h^2)}\\
        &\underset{h\to0}{=}\frac{1}{\sin^2(\theta)}\left[1-2\cot(\theta)h+\left(1+3\cot^2(\theta)\right)h^2+o(h^2)\right]\\
        &\underset{h\to0}{=}\frac{1}{\sin^2(\theta)}-\frac{2\cot(\theta)}{\sin^2(\theta)}h+\frac{1+3\cot^2(\theta)}{\sin^2(\theta)}h^2+o(h^2).
    \end{align*}
    Combining both, one also gets
    \begin{align*}
        \frac{\cos(\theta)-\cos(\theta+h)}{\sin^2(\theta+h)}&\underset{h\to0}{=}\left[\frac{1}{\sin^2(\theta)}-\frac{2\cot(\theta)}{\sin^2(\theta)}h+\frac{1+3\cot^2(\theta)}{\sin^2(\theta)}h^2+o(h^2)\right]\left[\sin(\theta)h+\tfrac{1}{2}\cos(\theta)h^2+o(h^2)\right]\\
        &\underset{h\to0}{=}\frac{h}{\sin(\theta)}+\left(\frac{\cos(\theta)}{2\sin^2(\theta)}-\frac{2\cot(\theta)}{\sin(\theta)}\right)h^2+o(h^2)\\
        &\underset{h\to0}{=}\frac{h}{\sin(\theta)}-\frac{3\cos(\theta)}{2\sin^2(\theta)}h^2+o(h^2).
    \end{align*}
    $\blacktriangleright$ \textit{Case $\mu>0$ :} For $\mu$ small enough, we have $\theta_{\mathbf{x}_1}\in(\theta_{k_0},\theta_{k_0+1}).$ Therefore, we can apply Lemma \ref{lem zonal} together with \eqref{phidot-dG} and write
    $$\dot{\Phi}_{\star}(\theta_{\mathbf{x}_1})-\dot{\Phi}(\theta_{\mathbf{x}_0})=\left(\frac{1}{\sin^2(\theta_{\mathbf{x}_1})}-\frac{1}{\sin^2(\theta_{\mathbf{x}_0})}\right)C(k_0)+\frac{\omega_{k_0+1}}{\sin^2(\theta_{\mathbf{x}_1})}\big(\cos(\theta_{\mathbf{x}_0})-\cos(\theta_{\mathbf{x}_1})\big),$$
    where we denoted
    $$C(k_0)\triangleq\sum_{k=1}^{k_0}\omega_k\big(\cos(\theta_{k-1})-\cos(\theta_k)\big).$$
    Using the Taylor expansions done at the beginning of the proof with $\theta=\theta_{k_0}=\theta_{\mathbf{x}_0}$ and $h=\mu$, we infer
    \begin{align*}
        \dot{\Phi}_{\star}(\theta_{\mathbf{x}_1})-\dot{\Phi}(\theta_{\mathbf{x}_0})&\underset{\mu\to0}{=}\left(\frac{\omega_{k_0+1}}{\sin(\theta_{\mathbf{x}_0})}-\frac{2C(k_0)\cot(\theta_{\mathbf{x}_0})}{\sin^2(\theta_{\mathbf{x}_0})}\right)\mu\\
        &\quad+\left(\frac{1+3\cot^2(\theta_{\mathbf{x}_0})}{\sin^2(\theta_{\mathbf{x}_0})}C(k_0)-\frac{3\omega_{k_0+1}\cos(\theta_{\mathbf{x}_0})}{2\sin^2(\theta_{\mathbf{x}_0})}\right)\mu^2+o(\mu^2)\\
        &\underset{\mu\to0}{=}\frac{1}{\sin^3(\theta_{\mathbf{x}_0})}\left(\omega_{k_0+1}\sin^2(\theta_{\mathbf{x}_0})-2C(k_0)\cos(\theta_{\mathbf{x}_0})\right)\mu\\
        &\quad+\frac{1}{2\sin^{4}(\theta_{\mathbf{x}_0})}\Big(2C(k_0)\big(\sin^2(\theta_{\mathbf{x}_0})+3\cos^2(\theta_{\mathbf{x}_0})\big)-3\omega_{k_0+1}\sin^2(\theta_{\mathbf{x}_0})\cos(\theta_{\mathbf{x}_0})\Big)\mu^2+o(\mu^2).
    \end{align*}
    First observe that the monotonicity condition \eqref{monotone cond} together with \eqref{order thetas} imply
    $$\omega_{k_0+1}=0\qquad\Rightarrow\qquad C(k_0)\neq0.$$
    By contraposition,
    $$C(k_0)=0\qquad\Rightarrow\qquad\omega_{k_0+1}\neq0.$$
    So 
    $$C(k_0)=0\qquad\Rightarrow\qquad\dot{\Phi}_{\star}(\theta_{\mathbf{x}_1})-\dot{\Phi}(\theta_{\mathbf{x}_0})\underset{\mu\to0}{\sim}\frac{\omega_{k_0+1}\mu}{\sin(\theta_{\mathbf{x}_0})},$$
    which allows to conclude. Therefore, in what follows, we assume 
    $$C(k_0)\neq0.$$
    Similarly, we get
    $$\omega_{k_0+1}\sin^2(\theta_{\mathbf{x}_0})-2C(k_0)\cos(\theta_{\mathbf{x}_0})\neq0\quad\Rightarrow\quad\dot{\Phi}_{\star}(\theta_{\mathbf{x}_1})-\dot{\Phi}(\theta_{\mathbf{x}_0})\underset{\mu\to0}{\sim}\frac{\mu}{\sin^3(\theta_{\mathbf{x}_0})}\left(\omega_{k_0+1}\sin^2(\theta_{\mathbf{x}_0})-2C(k_0)\cos(\theta_{\mathbf{x}_0})\right)$$ 
    and we obtain the desired result. Now assume that we are in the situation where $$\omega_{k_0+1}\sin^2(\theta_{\mathbf{x}_0})=2C(k_0)\cos(\theta_{\mathbf{x}_0}).$$
    We get the asymptotic
    $$\dot{\Phi}_{\star}(\theta_{\mathbf{x}_1})-\dot{\Phi}(\theta_{\mathbf{x}_0})\underset{\mu\to0}{\sim}\frac{C(k_0)\mu^2}{\sin^2(\theta_{\mathbf{x}_0})},$$
    which once again allows to conclude.\\
    $\blacktriangleright$ \textit{Case $\mu<0$ :} For $|\mu|$ small enough, we have $\theta_{\mathbf{x}_1}\in(\theta_{k_0-1},\theta_{k_0}).$ Therefore, we can apply Lemma \ref{lem zonal} together with \eqref{phidot-dG} and write
    $$\dot{\Phi}_{\star}(\theta_{\mathbf{x}_1})-\dot{\Phi}(\theta_{\mathbf{x}_0})=\frac{1}{\sin^2(\theta_{\mathbf{x}_1})}C(k_0-1)+\frac{\omega_{k_0}}{\sin^2(\theta_{\mathbf{x}_1})}\big(\cos(\theta_{k_0-1})-\cos(\theta_{\mathbf{x}_1})\big)-\frac{1}{\sin^2(\theta_{\mathbf{x}_0})}C(k_0).$$
    Notice that
    $$C(k_0)=C(k_0-1)+\omega_{k_0}\big(\cos(\theta_{k_0-1})-\cos(\theta_{k_0})\big).$$
    Hence, we obtain the new formula
    $$\dot{\Phi}_{\star}(\theta_{\mathbf{x}_1})-\dot{\Phi}(\theta_{\mathbf{x}_0})=\left(\frac{1}{\sin^2(\theta_{\mathbf{x}_1})}-\frac{1}{\sin^2(\theta_{\mathbf{x}_0})}\right)C(k_0-1)+\frac{\omega_{k_0}}{\sin^2(\theta_{\mathbf{x}_1})}\big(\cos(\theta_{\mathbf{x}_0})-\cos(\theta_{\mathbf{x}_1})\big).$$
    Using the Taylor expansions done at the beginning of the proof with $\theta=\theta_{k_0}=\theta_{\mathbf{x}_0}$ and $h=\mu$, we infer
    \begin{align*}
        \dot{\Phi}_{\star}(\theta_{\mathbf{x}_1})-\dot{\Phi}(\theta_{\mathbf{x}_0})&\underset{\mu\to0}{=}\frac{1}{\sin^3(\theta_{\mathbf{x}_0})}\left(\omega_{k_0}\sin^2(\theta_{\mathbf{x}_0})-2C(k_0-1)\cos(\theta_{\mathbf{x}_0})\right)\mu\\
        &\quad+\frac{1}{2\sin^{4}(\theta_{\mathbf{x}_0})}\Big(2C(k_0-1)\big(\sin^2(\theta_{\mathbf{x}_0})+3\cos^2(\theta_{\mathbf{x}_0})\big)-3\omega_{k_0}\sin^2(\theta_{\mathbf{x}_0})\cos(\theta_{\mathbf{x}_0})\Big)\mu^2+o(\mu^2).
    \end{align*}
    At this point, we can proceed as before with $C(k_0)\leadsto C(k_0-1)$ and $\omega_{k_0+1}\leadsto\omega_{k_0}$. This achieves the proof of Lemma \ref{lem alpha non zero}.
\end{proof}
We parametrize the curve $\Gamma(0)$ by
$$\gamma:[0,1]\longrightarrow \Gamma(0),\qquad\gamma(0)=\mathbf{x}_0,\qquad\gamma(1)=\mathbf{x}_1.$$ 
Then, we set for any $t\geqslant0$ and any $s\in[0,1],$
$$\gamma_t(s)\triangleq\phi(t,\gamma(s))\qquad\textnormal{so that}\qquad\forall\ell\in\{0,1\},\quad\phi(t,\mathbf{x}_{\ell})=\gamma_t(\ell).$$
Our aim is to prove the following result giving a lower bound on the length of the curve $\gamma_{T}([0,1]).$ 
\begin{prop}\label{prop length}
Let $\mu_0$ as in Lemma \ref{lem alpha non zero}. Then, for any $\mu\in(-\mu_0,\mu_0)$, there exist $\kappa\triangleq\kappa(\mu)>0$ and $T_0\triangleq T_0(\mu)>0$ such that for any $T\geqslant T_0,$ there exists $\overline{\delta}\triangleq\overline{\delta}(\mu,\mathtt{M},T)>0$ such that if
\begin{equation}\label{small length}
    \|\zeta_0-\zeta_{\star}\|_{L^1(\mathbb{S}^2)}<\overline{\delta},
\end{equation}
then we have
$$\textnormal{Length}\big(\gamma_T([0,1])\big)\geqslant\kappa(T-T_0).$$
\end{prop}
\begin{proof}
    Let us recall that on the sphere, the length is given by
$$\textnormal{Length}\big(\gamma_T([0,1])\big)=\int_{0}^{1}\sqrt{\mathtt{g}_{\mathbb{S}^2}\big(\dot{\gamma}_T(s),\dot{\gamma}_T(s)\big)_{\gamma_T(s)}}ds,$$
where $\mathtt{g}_{\mathbb{S}^2}$ is the metric introduced in \eqref{def metric}.
By construction, for any $\mathbf{x}\in\mathbb{S}^2$ and any $(\alpha_1,\alpha_2,\beta_1,\beta_2)\in\mathbb{R}^4,$ we have
$$\mathtt{g}_{\mathbb{S}^2}\big(\alpha_1\mathbf{e}_{\theta}(\mathbf{x})+\beta_1\mathbf{e}_{\varphi}(\mathbf{x}),\alpha_2\mathbf{e}_{\theta}(\mathbf{x})+\beta_2\mathbf{e}_{\varphi}(\mathbf{x})\big)_{\mathbf{x}}=\alpha_1\alpha_2+\beta_1\beta_2.$$
Besides, differentiating \eqref{flow in local chart}, we get
\begin{align*}
    \dot{\gamma}_T(s)&=\partial_{s}\big[\phi\big(T,\gamma(s)\big)\big]\\
    &=\partial_{s}\left[\psi_1\Big(\Theta\big(T,\gamma(s)\big),\Phi\big(T,\gamma(s)\big)\Big)\right]\\
    &=\partial_s\left[\Theta\big(T,\gamma(s)\big)\right]\partial_{\theta}\psi_1\Big(\Theta\big(T,\gamma(s)\big),\Phi\big(T,\gamma(s)\big)\Big)+\partial_s\left[\Phi\big(T,\gamma(s)\big)\right]\partial_{\varphi}\psi_1\Big(\Theta\big(T,\gamma(s)\big),\Phi\big(T,\gamma(s)\big)\Big)\\
    &=\partial_s\left[\Theta\big(T,\gamma(s)\big)\right]\mathbf{e}_{\theta}\Big(\phi\big(T,\gamma(s)\big)\Big)+\sin\Big(\Theta\big(T,\gamma(s)\big)\Big)\partial_s\left[\Phi\big(T,\gamma(s)\big)\right]\mathbf{e}_{\varphi}\Big(\phi\big(T,\gamma(s)\big)\Big).
\end{align*}
We deduce that
$$\textnormal{Length}\big(\gamma_T([0,1])\big)=\int_{0}^{1}\sqrt{\Big(\partial_s\left[\Theta\big(T,\gamma(s)\big)\right]\Big)^2+\sin^2\Big(\Theta\big(T,\gamma(s)\big)\Big)\Big(\partial_s\left[\Phi\big(T,\gamma(s)\big)\right]\Big)^2}ds.$$
Let us take 
$$\theta_{\min}(\mu)\triangleq\theta_{1}-|\mu|,\qquad\theta_{\max}(\mu)\triangleq\theta_{N-1}+|\mu|$$
and
$$\xi(\mu)\triangleq\min\Big(\tfrac{|\mu|}{2}\sin^2\big(\theta_{\min}(\mu)\big),\tfrac{|\mu|}{2}\sin^2\big(\theta_{\max}(\mu)\big),\xi_0\big(1,\theta_{\min}(\mu),\theta_{\max}(\mu)\big)\Big),$$
where $\xi_0$ is defined in Proposition \ref{prop confinement}. If $\mu_0$ is small enough, then $\theta_{\min}(\mu)>0$ and $\theta_{\max}(\mu)<\pi.$
We choose 
\begin{equation}\label{choice delta-bar}
    \overline{\delta}(\mu,\mathtt{M},T)\triangleq\delta_2\big(\xi(\mu),\mathtt{M},T),
\end{equation}
where $\varepsilon\mapsto\delta_2(\varepsilon)$ is defined in Proposition \ref{prop confinement}. The smallness assumption \eqref{small length} together with the choice \eqref{choice delta-bar} and the property \eqref{confined init curve} allow to apply Proposition \ref{prop confinement} and get that for any $s\in[0,1]$ and if $T\geqslant1,$
\begin{equation}\label{confin length}
    \left|\Theta(T,\gamma(s)\big)-\theta_{\gamma(s)}\right|\leqslant\xi(\mu)\qquad\textnormal{and}\qquad\left|\Phi\big(T,\gamma(s)\big)-\big(\Phi_{\star}(\theta_{\gamma(s)})T+\varphi_{\gamma(s)}\big)\right|\leqslant\frac{|\mu|}{2}\cdot
\end{equation}
The first condition in \eqref{confin length} together with the fact that $\sin$ is $1$-Lipschitz and the choice of $\xi(\mu)$ imply that if $T\geqslant1,$ we get
$$\inf_{s\in[0,1]}\sin\Big(\Theta\big(T,\gamma(s)\big)\Big)\geqslant\frac{1}{2}\min\Big(\sin\big(\theta_{\min}(\mu)\big),\sin\big(\theta_{\max}(\mu)\big)\Big)\triangleq \beta(\mu)>0.$$
Consequently,
\begin{equation}\label{bnd1 legnth}
    \begin{aligned}
    \textnormal{Length}\big(\gamma_T([0,1])\big)&\geqslant \beta(\mu)\int_{0}^{1}\left|\partial_{s}\left[\Phi\big(T,\gamma(s)\big)\right]\right|ds\\
    &\geqslant \beta(\mu)\left|\int_{0}^{1}\partial_{s}\Phi\big(T,\gamma(s)\big)ds\right|\\
    &=\beta(\mu)\left|\Phi\big(T,\gamma(1)\big)-\Phi\big(T,\gamma(0)\big)\right|\\
    &=\beta(\mu)\left|\Phi(T,\mathbf{x}_1)-\Phi(T,\mathbf{x}_0)\right|.
\end{aligned}
\end{equation}
Applying the second condition in \eqref{confin length} with $s\in\{0,1\}$, we get
$$\alpha(\mu)\big(T-T_{-}(\mu)\big)\leqslant\Phi(T,\mathbf{x}_1)-\Phi(T,\mathbf{x}_0)\leqslant\alpha(\mu)\big(T-T_+(\mu)\big),$$
where
$$\alpha(\mu)\triangleq\dot{\Phi}_{\star}\big(\theta_{\mathbf{x}_1}(\mu)\big)-\dot{\Phi}_{\star}(\theta_{\mathbf{x}_0}),\qquad T_{\pm}(\mu)\triangleq\frac{\varphi_{\mathbf{x}_0}-\varphi_{\mathbf{x}_1}}{\dot{\Phi}_{\star}\big(\theta_{\mathbf{x}_1}(\mu)\big)-\dot{\Phi}_{\star}(\theta_{\mathbf{x}_0})}\mp\frac{|\mu|}{\dot{\Phi}_{\star}\big(\theta_{\mathbf{x}_1}(\mu)\big)-\dot{\Phi}_{\star}(\theta_{\mathbf{x}_0})}\cdot$$
Let us mention that, by virtue of Lemma \ref{lem alpha non zero}, $\alpha(\mu)\neq0$ for $\mu\in(-\mu_0,\mu_0).$ We denote
$$\kappa(\mu)\triangleq \beta(\mu)|\alpha(\mu)|>0\qquad\textnormal{and}\qquad T_0(\mu)\triangleq\max\big(1,T_{-}(\mu),T_{+}(\mu)\big)>0.$$
We deduce that for $T\geqslant T_0(\mu),$ we have
$$\left|\Phi(T,\mathbf{x}_1)-\Phi(T,\mathbf{x}_0)\right|\geqslant|\alpha(\mu)|\min\big(T-T_-(\mu),T-T_+(\mu)\big)\geqslant|\alpha(\mu)|\big(T-T_0(\mu)\big).$$
Plugging this into \eqref{bnd1 legnth} gives
$$\textnormal{Length}\big(\gamma_T([0,1])\big)\geqslant\kappa(\mu)\big(T-T_0(\mu)\big).$$
This ends the proof of Proposition \ref{prop length}.
\end{proof}
\begin{remark}
    We end this document by mentioning that our analysis shows that the orientation of the filament depends asymptotically on the sign of $\alpha(\mu).$ This latter can be tracked along the proof of Lemma \ref{lem alpha non zero}.
\end{remark}

	\vspace{2cm}
	\noindent GIAN MARCO MARIN: School of Mathematics, Georgia Institute of Technology, Atlanta GA.\\
	E-mail address: \href{mailto:gmarin8@gatech.edu}{\texttt{gmarin8@gatech.edu}}\\
    
	\noindent EMERIC ROULLEY: SISSA International School for Advanced Studies, Via Bonomea 265 Trieste, 34136 Italy.\\
	E-mail address: \href{mailto:eroulley@sissa.it}{\texttt{eroulley@sissa.it}}

\begin{thebibliography}{9999}

    \bibitem{BD15} S. Boatto, D. G. Dritschel, \textit{The motion of point vortices on closed surfaces,} Proceedings of the Royal Society A - Mathematical, Physical and Engineering Sciences  471 (2015), no. 2176:20140890, 25. 

    \bibitem{CLW24} D. Cao, S. Li, G. Wang, \textit{Desingularization of vortices for the incompressible Euler equation on a sphere,} \href{http://arxiv.org/abs/2411.07645}{arXiv:2411.07645}.

    \bibitem{CWZ23} D. Cao, G. Wang, B. Zuo, \textit{Stability of degree-2 Rossby-Haurwitz waves,} \href{http://arxiv.org/abs/2305.03279}{arXiv:2305.03279}.
    
    \bibitem{CM88} S. Caprino, C. Marchioro, \textit{On nonlinear stability of stationary Euler flows on a rotating sphere,} Journal of Mathematical Analysis and Applications 129 (1988), no. 1, 24--36.

    \bibitem{CJ21} K. Choi, I.-J. Jeong, \textit{Growth of perimeter for vortex patches in a bulk,} Applied Mathematics Letters 113 (2021), no. 106857.

    \bibitem{CJ22} K. Choi, I.-J. Jeong, \textit{Stability and instability of Kelvin waves,} Calculus of Variations 61 (2022), no. 221.

    \bibitem{CJ23} K. Choi, I.-J. Jeong, \textit{Filamentation near Hill's vortex,} Communications in Partial Differential Equations 48 (2023), no. 1.

    \bibitem{CDG24} A. Constantin, D. G. Dritschel, P. Germain, \textit{The onset of filamentation on vorticity interfaces in two-dimensional Euler flows,} \href{http://arxiv.org/abs/2410.09610}{arXiv:2410.09610}.
 
    \bibitem{CG22} A. Constantin, P. Germain, \textit{Stratospheric planetary flows from the perspective of the Euler equation on a rotating sphere,} Archive for Rational Mechanics and Analysis 245 (2022), no. 1, 587--644.

    \bibitem{DP92} D. G. Dritschel, L. M. Polvani, \textit{The roll-up of vorticity strips on the surface of a sphere,} Journal of Fluid Mechanics 234 (1992), 47--69.

    \bibitem{DP93} D. G. Dritschel, L. M. Polvani, \textit{Wave and vortex dynamics on the surface of a sphere,} Journal of Fluid Mechanics 255 (1993), 35--64. 

    \bibitem{GHR23} C. Garc\'{i}a, Z. Hassainia, E. Roulley, \textit{Dynamics of vortex cap solution on the rotating unit sphere,} Journal of Differential Equations 417 (2025), 1--63.

    \bibitem{HH04} Gregory J. Hakim, James R. Holton, \textit{An introduction to dynamic meteorology,} International Geophysics Series. Elsevier Academic Press, Burlington, MA, 4 edition, 2004.

    \bibitem{KSS18} S.-C. Kim, T. Sakajo, S.-I. Sohn, \textit{Stability of barotropic vortex strip on a rotating sphere,} Proceedings A 474 (2018), no. 2210-20170883.
    
    \bibitem{KS21} S.-C. Kim, S.-I Sohn, \textit{Linear stability and nonlinear evolution of a polar vortex cap on a rotating sphere,} European Journal of Mechanics. B. Fluids 85 (2021), 102--109.

    \bibitem{SZ24} T. Sakajo, C. Zou, \textit{Regularization for point vortices on $\mathbb{S}^2$,} \href{http://arxiv.org/abs/2411.11388}{arXiv:2411.11388}.

    \bibitem{SZ24-1} T. Sakajo, C. Zou, \textit{$C^1$ type regularization for point vortices on $\mathbb{S}^2$,} \href{http://arxiv.org/abs/2411.15176}{arXiv:2411.15176}.
\end{thebibliography}
\end{document}